\documentclass[11pt,reqno]{amsart} 
\usepackage[utf8]{inputenc}
\usepackage{amsfonts,amsmath,amsthm,amssymb}
\usepackage[justification=centering]{caption}
\usepackage{color}
\usepackage{fullpage}
\usepackage{enumitem}   
\usepackage{hyperref,cleveref, float}
\usepackage{mathtools,etaremune}
\usepackage{xcolor}
\usepackage{thm-restate}
\usepackage{comment}
\usepackage{xurl}
\usepackage{subcaption}
\usepackage{thmtools}
\usepackage{hyperref}
\usepackage{cleveref}			
\usepackage{dsfont}
\usepackage{enumitem}
\usepackage{tikz}
\newtheorem{theorem}{Theorem}[section]
\newtheorem{lemma}[theorem]{Lemma}
\newtheorem{corollary}[theorem]{Corollary}
\newtheorem{proposition}[theorem]{Proposition}
\theoremstyle{definition}
\newtheorem{example}[theorem]{Example}  
\newtheorem{definition}[theorem]{Definition}   

\newcommand{\Fub}{\mathrm{FR}}
\newcommand{\fr}{\mathrm{fr}}

\newcommand{\Sym}{\mathfrak{S}}

\newcommand*{\NN}{\mathbb{N}}

\newcommand*{\x}{\mathbf{x}}

\newcommand{\FR}{\mathrm{\mathrm{FR}}}
\newcommand{\seqnum}[1]{\href{http://oeis.org/#1}{\underline{#1}}}

\newcommand{\NDFRn}{\mathrm{\mathrm{FR}_n^\uparrow}}
\newcommand{\NIFRn}{\mathrm{\mathrm{FR}_n^\downarrow}}

\newcommand{\flatn}{\textnormal{flat}}

\newtheorem{conjecture}[theorem]{Conjecture}

\tikzstyle{vertex}=[circle, draw, fill=black, inner sep=0pt, minimum size=1pt]

\title{Enumerating Flat Fubini Rankings}
\author[K.~Barrese]{Kenny Barrese}
\address[K.~Barrese]{Mathematics and Natural Science Division, Brescia University, Owensboro, KY, 42301
}
\email{\textcolor{blue}{\href{mailto:kenny.barrese@brescia.edu}{kenny.barrese@brescia.edu}}}

\author[Elder]{Jennifer Elder}
\address[J.~Elder]{Department of Computer Science, Mathematics and Physics, Missouri Western State University, St. Joseph, MO, 64507}
\email{\textcolor{blue}{\href{mailto:jelder@missouriwestern.edu}{jelder@missouriwestern.edu}}}

\author[Harris]{Pamela E. Harris}
\address[P.~E.~Harris]{Department of Mathematical Sciences, University of Wisconsin, Milwaukee, WI, 53211}
\email{\textcolor{blue}{\href{mailto:peharris@uwm.edu}{peharris@uwm.edu}}}

\author[Simpson]{Anthony Simpson}
\address[A.~Simpson]{Greenwood, IN 46143}
\email{\textcolor{blue}{\href{mailto:alsimpson1996@gmail.com}{alsimpson1996@gmail.com}}}

\newcommand{\run}{\texttt{run}}
\newcommand{\wrun}{\texttt{wrun}}

 \newcommand{\flatwruns}{\texttt{flat\textunderscore wruns}}
 
 \newcommand{\flatruns}{\texttt{flat\textunderscore runs}}
 \newcommand{\wflatruns}{\texttt{wflat\textunderscore runs}}
 \newcommand{\wflatwruns}{\texttt{wflat\textunderscore wruns}}

\begin{document}

\begin{abstract}
    Recall that the set of Fubini rankings on $n$ competitors consists of the $n$-tuples that encode the possible rankings of $n$ competitors in a competition allowing ties. 
    Moreover, recall that a run (weak run) in a tuple is a subsequence of consecutive ascents (weak ascents). 
    If the leading terms of the set of maximally long runs (weak runs) of a tuple are in increasing (weakly increasing) order, then the tuple is said to be flattened (weakly flattened).
    We define the set of strictly flattened Fubini rankings, which is the subset of Fubini rankings with runs of strict ascents whose leading term are strictly increasing. Analogously, we define the set of weakly flattened Fubini rankings, which is the subset of Fubini rankings with runs of weak ascents whose leading terms are in weakly increasing order. Our main results give formulas for the enumeration of strictly flattened Fubini rankings and weakly flattened Fubini rankings. We also provide some conjectures for further study.
\end{abstract}

\let\thefootnote\relax\footnotetext{\textit{2020 Mathematics Subject Classification}. Primary 05A05; Secondary 05A10, 05A15,  05A17, 05A18.}
\let\thefootnote\relax\footnotetext{\textit{Key words and phrases}. Fubini ranking, flattened word, flattened Fubini ranking, content of a word}

\maketitle

\section{Introduction}
Let $\NN$ denote the set of positive integers and let $[n]=\{1,2,\ldots, n\}$ whenever $n\in\NN$. 
Throughout we let $\Sym_n$ denote the set of permutations which we write in one-line notation $\pi=\pi_1\pi_2\cdots\pi_n$. 
We recall that a set partition $\beta$ of $[n]$ is a collection of subsets (also called blocks) $\beta_1,\beta_2,\ldots,\beta_k$ satisfying the conditions that 
$\cup_{i=1}^k\beta_i=\beta$ and $\beta_i\cap \beta_j=\emptyset$ for all $i\neq j$.
In \cite{callan}, Callan studied the Mathematica Flatten function, which we denote by  $\flatn$, whose input is an ordered set partition (where the order of the blocks matter) and returns a permutation by deleting the block notation and concatenating the numbers in a single row. 
For example, under the flatten operation, the set partition $\Pi=1|36|524|8|79$ will become
\[
\text{Flatten}(\Pi) = 136524879.
\]
One way to uniquely order the blocks of a set partition is to order blocks by their minimal elements, and then order the elements within a block in increasing order. Whenever this ordering is implemented, the flattened function returns a permutation whose runs of ascents have the property that the leading terms of those runs are in increasing order. Using the previous example, we would reorder the blocks and elements in the blocks as $\Pi=1|245|36|79|8$, which produces
\[
    \text{Flatten}(\Pi)=124536798.
\]
A permutation with this property is often referred to as a flattened permutation or flattened partition \cite{callan}. 
Callan considered the set of flattened permutations which also avoid certain patterns of length three, whose enumerations involve powers of 2, Fibonacci numbers, Catalan numbers, and binomial transforms of Catalan numbers.

Nabawanda, Rakotondrajao, Bamunoba \cite{ONFRAB} and Beyene, Mantaci \cite{Beyene}, independently gave formulas for the number of flattened permutations of length $n$ with $k$ runs. They showed multiple recursive identities, including
\begin{equation*}
f_{n,k} = \sum_{m=1}^{n-2} \left (\binom{n-1}{m}-1 \right)f_{m,k-1}.
\end{equation*}
This work has since been expanded to a subset of parking functions which consist of a permutation on the multiset $
\{1^{r+1},2,3,4,\ldots,n\}$, where $r+1$ denotes the multiplicity of $1$ \cite{flat_pf}. 
Baril, Harris, and Ram\'{i}rez studied flattened Catalan words and some related statistics \cite{baril2024flattenedcatalanwords}, and 
Buck, Elder, Figueroa, Harris, Harry, and Simpson studied
flattened Stirling permutations \cite{flat_sp}, giving enumerative formulas and bijections to other combinatorial objects such as type $B$ set partitions. 

We consider the analogous problem of enumerating flattened Fubini rankings. 
We recall that a \textit{Fubini ranking of length $n$} is a tuple $r=(r_1,r_2,\ldots,r_n)\in [n]^n$ that records a valid ranking over $n$ competitors with ties allowed (i.e., multiple competitors can be tied and have the same rank).
However, if $k$ competitors are tied and rank $i$th, then the $k-1$ subsequent ranks $i + 1, i + 2, \ldots, i + k - 1$ are disallowed.  
We let $\FR_n$ denote the set of all Fubini rankings on $n$ competitors and let $\fr_n=|\FR_n|$ denote its cardinality. 
If $n=3$, then 
\[\FR_3=\left\{\begin{matrix}
    (1,2,3),(1,3,2),(2,1,3),(2,3,1),(3,1,2),(3,2,1),\\(1,1,1),(1,1,3),(1,3,1),(3,1,1),(1,2,2),(2,1,2),(2,2,1),
\end{matrix}\right\}\]
and $\fr_3=13$.
It was shown by Cayley \cite{cayley_2009} that the number of Fubini rankings is given by the $n$th Fubini number \cite[\seqnum{A000670}]{OEIS}
\[\fr_{n}=\sum _{{k=0}}^{n}\sum _{{j=0}}^{k}(-1)^{{k-j}}{\binom  {k}{j}}j^{n}.\label{fubini numbers}\]

We remark that Fubini rankings have received recent attention in the literature due to their connection to parking functions. Harris and Haddaway established that Fubini rankings are in bijection with unit interval parking functions, see \cite{hadaway2022combinatorial}, thereby establishing that the number of unit interval parking functions is also a Fubini number. 
Utilizing the connection between unit interval parking functions and Fubini rankings, Brandt, Elder, Harris, Rojas Kirby, Reutercrona, Wang, and Whidden~\cite{unit_pf} gave an
identity for the Fubini numbers as a sum of multinomials over compositions. 
They also considered a generalization of Fubini rankings, called the $r$-Fubini rankings of length $n+r$, whose first $r$ values are distinct.
Their main result established a bijection between $r$-Fubini rankings and unit interval parking functions of length $n+r$ where the first $r$ cars have distinct preferences. 
 This showed that those parking functions are also enumerated by the $r$-Fubini numbers.
 In other work, Elder, Harris, Kretschmann, and Mart\`inez Mori \cite{elder2023boolean} studied the subset of Fubini rankings called \textit{unit Fubini rankings}, which satisfy that at most two competitors tie for any single rank.  
 A main contribution of Elder, Harris, Kretschman, and Mart\'inez Mori \cite{elder2023boolean}, established that the unit Fubini rankings with exactly $n-k$ distinct ranks are in bijection with Boolean intervals of rank $k$ in the weak (Bruhat) order of the symmetric group $\Sym_{n}$.

In our work, we consider the natural question lying at the intersection of these themes: \textit{How many Fubini rankings on $n$ competitors are flattened?}
Given that Fubini rankings are a super set of the set of permutations, we first note that defining a word to be flattened requires two inequality conditions: a condition on the definition of the runs as strict ascents or weak ascents, and a condition 
on the leading terms to either be strictly increasing or weakly increasing. 
We follow the definitions in \cite{flat_pf}, and define \textit{weakly flattened Fubini rankings with runs of weak ascents} to be a Fubini ranking with runs of weak ascents and whose leading terms are in weakly increasing order. 
For example, $(1,3,3,1)$ is a weakly flattened Fubini ranking with runs of weak ascents that contains two runs: $1,3,3$ and $1$. The run $1,3,3$ is allowed because we are dealing with runs of weak ascents and the two runs are allowed to both begin with $1$ because the leading terms are required to be in weakly increasing order.
We also define \textit{flattened Fubini rankings with runs of ascents}, which are Fubini rankings whose runs are made of strict ascents and whose leading terms are in strictly increasing order. 
For example, $(1,2,2,4)$ contains two runs of (strict) ascents, $1,2$ and $2,4$, whose leading terms are in (strictly) increasing order.
Note that in this way all strictly flattened Fubini rankings are weakly flattened, but the converse is not true. For the rest of this paper, if the flattness condition or the runs are not specified to be strict or weak, strict is assumed to be the default.

Much of our work utilizes the \emph{content of a Fubini ranking}, and the \emph{reduced content of a Fubini Ranking}. The content of a Fubini ranking is the traditional definition of content: for example, the element $(1,3,1,3,5,6)\in \FR_6$ has content $(2,0,2,0,1,1)$, where each entry $x_i$ in the content represents the number of times $i$ appears. In Section~\ref{sec:weakflat}, we discuss the concept of reduced content which removes all the 0's from our content vector, and how we can recover the content of a Fubini ranking based on the reduced content.

Our main contributions are as follows:

\begin{enumerate}
    \item Lemma~\ref{lem:Fib}:
     Let $\NDFRn$ denote the set of weakly increasing Fubini rankings and $\NIFRn$ denote the set of weakly decreasing Fubini rankings.  If $n\geq 1$, then 
    \[|\Fub_n^\uparrow|=|\Fub_n^\downarrow|=2^{n-1}.\]

    \item \Cref{wwenum}: Let $\textbf{a}= (a_1,a_2,\ldots,a_k)\in\mathbb{Z}_{\geq 1}^n$ be a composition of $n$. Then the number of  weakly flattened Fubini rankings with runs of weak ascents that have reduced content $\textbf{a}$ is:
\footnotesize{$$ \hspace{0.3in} \sum_{j_2=0}^{a_2}\left [\binom{a_1+a_2-j_2-2}{a_2-j_2}\sum_{j_3=0}^{a_3}\left [\binom{a_1+j_2+a_3-j_3-2}{a_3-j_3}\sum_{j_4=0}^{a_4}\left [\ldots \sum_{j_k=0}^{a_k}\left [\binom{a_1+a_k-j_k-2+\displaystyle\sum_{r=2}^{k-1}j_{r}}{a_k-j_k}\right ]\ldots\right ]\right ]\right ].$$}

    \item \normalsize{Corollary~\ref{cor:ends_in_1}: If $\textbf{a}= (a_1,a_2,\ldots,a_k)\in\mathbb{Z}_{\geq 1}^n$ is a composition of $n$, then
the number of weakly flattened Fubini rankings with runs of weak ascents that have reduced content $\textbf{a}$ ending in an entry with value 1 is:
    $$\prod_{i=2}^k\binom{a_1+a_i-2}{a_i}.$$}
    \item Theorem~\ref{thm:wfFub}: Let $|\wflatwruns_k(\FR_n)|$ denote the number of weakly flattened Fubini rankings of length $n>0$ with $k>0$ runs of weak ascents. Then $|\wflatwruns_k(\FR_{j})=0|$, exactly when $k>\lceil \frac{j}{2}\rceil$ for any $j\in \mathbb{Z}_{\geq 1}$.
    \item Proposition~\ref{prop:end_in_1}: Let 
$\textbf{a}=(a_1,a_2,\ldots,a_n)$ denote the content of a Fubini ranking with $k\leq a_1\neq n-k+1$. Let
$F_{n,k}(\bf{a})$ be the number of weakly flattened Fubini rankings of length $n$ with content $\bf{a}$ which have $k$ runs of weak ascents and also end in the value $1$. Then 
    \[F_{n,k}(\textbf{a})=\binom{a_1-1}{k-1}\cdot|\mathcal{B}_k(\textbf{a})|,\]
    where $\mathcal{B}_k(\bf{a})$ is the set of 
$n\times (k-1)$ nonnegative integer valued matrices whose $i$th row is a weak composition of $a_i$ and whose column sums are always greater than or equal to $1$.
 \item \normalsize{Theorem~\ref{thm:sec6}: 
    Let ${\bf a} = (a_1,a_2,\ldots,a_n)$ be a weak composition of $n$ and the content of a Fubini ranking with $n$ competitors. Let $a$ be the median entry in a Fubini ranking, $f$, with content $\textbf{a}$, $m_a$ be the number of entries in $f$ strictly less than the median $a$, and $M_a$ be the number of entries in $f$ strictly greater than the median $a$. 
    Then exactly one of the following holds.
    \begin{enumerate}
        \item If $m_a = M_a = 0$, then the maximum possible number of runs in $f$ is $1$.
        \item If $0=m_a<M_a$, then the maximum possible number of runs in $f$ is $M_a+1$.
        \item If $0=M_a<m_a$, then the maximum possible number of runs in $f$ is $m_a$.
        \item If $m_a,M_a>0$, then the maximum possible number of runs in $f$ is the minimum of $m_a + M_a$ and $\lceil\frac{n}{2}\rceil$.
    \end{enumerate}}
    
    \item Theorem~\ref{sscontent}: 
Let ${\bf a} = (a_1,a_2,\ldots,a_k)\in \mathbb{Z}^k_{\geq 1}$ be a composition of $n$ with $k$ parts. Then there exists a flattened Fubini ranking with runs of ascents whose reduced content is $\textbf{a}$ if and only if $a_{i}\leq i$ for all $1\leq i\leq k$.
    \item Theorem~\ref{wsenum}: 
    Let $(a_1,a_2,\ldots,a_k)$, a composition of $n$ containing positive entries, be the reduced content of a Fubini ranking. 
    Then the number of weakly flattened Fubini rankings with runs of ascents having reduced content $(a_1,a_2,\ldots,a_k)$ is given by \footnotesize{$$\sum_{j_2=0}^{a_2}\left (\binom{a_1-1}{a_2-j_2}\sum_{j_3=0}^{a_3}\left (\binom{a_1+j_2-1}{a_3-j_3}\sum_{j_4=0}^{a_4}\left (\ldots \sum_{j_k=0}^{a_k}\left (\binom{a_1+j_2+j_3+\cdots+j_{k-1}-1}{a_k-j_k}\right )\ldots\right )\right )\right ).$$}
    
\end{enumerate}

This article is organized as follows. 
In Section~\ref{sec:prelim} we present the technical definitions to make our approach precise, as well as some preliminary results related to Fubini rankings. 
In Section~\ref{sec:weakflat} we establish results related to Fubini rankings that are weakly flattened with runs of weak ascents.
In Section~\ref{sec:strongflat}, we establish results related to Fubini rankings that are flattened with runs of ascents.
In Section~\ref{sec:mixed_conditions}, we establish results related to Fubini rankings with a mix of weak and strict restrictions.
We conclude with Section~\ref{sec:future}, which summarizes our open questions, conjectures, and future directions for research in this area.

\section{Preliminaries}\label{sec:prelim}

Given a word $w=w_1w_2\cdots w_n$ we say $w$ has a descent at $i$ if $w_{i}>w_{i+1}$ and an ascent at $i$ if $w_i<w_{i+1}$. 
A weak descent is an index $i$ where $w_i\geq  w_{i+1}$, a weak ascent is an index $i$ where $w_{i}\leq w_{i+1}$. 
We say a word is increasing if there are ascent at every index $i\in[n-1]$ and it is decreasing if there are descents at every $i\in[n-1]$. Whenever we say the word is weakly increasing, it means that the word has weak ascents at every index $i\in[n-1]$, and likewise we say the word is weakly decreasing if the word has weak descents at every index $i\in[n-1]$. Whenever $w_i=w_{i+1}$ we say $w$ has a tie at index $i$.

Given a word $w=w_1w_2\cdots w_n$ with entries in some totally ordered set $X$, usually we let $X=[n]$,  we can define 
\begin{itemize}
    \item \emph{runs of ascents} as the longest contiguous subwords $w_iw_{i+1}\cdots w_{j}$ such that $w_k<w_{k+1}$ for all $i\leq k\leq j-1$, and 
    \item \emph{runs of weak ascents} as the longest contiguous subwords $w_iw_{i+1}\cdots w_{j}$ such that $w_k\leq w_{k+1}$ for all $i\leq k\leq j-1$.
\end{itemize}
Regardless of  the makeup of the runs, which definition of runs we use, whether it is runs of ascents or runs of weak ascents, we define the following. 
\begin{itemize}
    \item The word $w$ is flattened if the leading terms of the runs are in increasing order, and 
    \item the word $w$ is weakly flattened if the leading terms of the runs are in weakly increasing order.
\end{itemize}
For example, 
$123245$ is a flattened word with runs of ascents,
and $123145233$ is weakly flattened with runs of weak ascents.
Baril, Harris, and Ramirez
\cite{baril2024flattenedcatalanwords} defined and studied flattened Catalan words as Catalan words whose runs of weak ascents have leading terms that appear in weakly increasing order.

Throughout our set of interest is the set of Fubini rankings of length $n$ and we use the following definitions:
\begin{itemize}
    \item Let $\flatruns(\FR_n)$ be the set of flattened Fubini rankings with runs of ascents.
    \item Let $\flatwruns(\FR_n)$ be the set of flattened Fubini rankings with runs of weak ascents.
    \item Let $\wflatruns(\FR_n)$ be the set of weakly flattened Fubini rankings with runs of ascents.
    \item Let $\wflatwruns(\FR_n)$ be the set of weakly flattened Fubini rankings with runs of weak ascents.  
    \item Whenever we need to specify the number of runs we use a subscript of $k$. For example, $\flatruns_k(\FR_n)$ is the set of flattened Fubini rankings with $k$ runs of ascents. The analogous definitions hold for $\flatwruns_k(\FR_n)$, $\wflatruns_k(\FR_n)$, and $\wflatwruns_k(\FR_n)$.
\end{itemize}
As usual, we use $|X|$ to denote the cardinality of a set $X$.
    





We begin by counting Fubini rankings that are either weakly increasing or weakly decreasing tuples. 

\begin{lemma}\label{lem:Fib} Let $\NDFRn$ denote the set of weakly increasing Fubini rankings and $\NIFRn$ denote the set of weakly decreasing Fubini rankings.  If $n\geq 1$, then 
    \[|\Fub_n^\uparrow|=|\Fub_n^\downarrow|=2^{n-1}.\]
\end{lemma}
\begin{proof}
Reutercrona, Wang, and Whiden, proved that $|\Fub_n^\uparrow|=2^{n-1}$ \cite{summer_report}, and they also studied parking functions with a given descent set. More on this has been considered in \cite{cruz2024discretestatisticsparkingfunctions}. By taking the reverse of every weakly increasing Fubini ranking, we arrive at a weakly decreasing Fubini ranking, hence the result follows.
\end{proof}
    
\begin{corollary}
When we restrict to Fubini rankings with a single run, we have \[|\emph{\wflatwruns}_1(\FR_n)| = |\emph{\flatwruns}_1(\FR_n)| = 2^{n-1}.\] 
\end{corollary}
\begin{proof}
The result follows from two observations. First, the set
${\wflatwruns}_1(\FR_n)$ and the set ${\flatwruns}_1(\FR_n)$
are equivalent to the set of weakly increasing Fubini rankings of length $n$. Second, as these rankings consist of a single weak run, weakly flattened and flattened agree. 
\end{proof}

\section{Weakly flattened Fubini rankings with runs of weak ascents}\label{sec:weakflat}
In this section, we consider the set of Fubini rankings of length $n$ with runs of weak ascents and whose leading terms are in weakly increasing order, denoted $\wflatwruns(\FR_n)$. 
To this end we introduce the following definition.

\begin{definition}\label{def:content_fub}
Given a Fubini ranking $(x_1,x_2,\ldots,x_n)\in\FR_n$ we define the \emph{content} of $\x$ by $\textbf{a}=(a_1,a_2,\ldots,a_n)$ where $a_i$ denotes the number of competitors with rank $i$. 
\end{definition}

Note that the content gives the multiplicity of each rank. Moreover if we know what ranks appear and the length of the Fubini ranking, then we know the content.

\begin{example}
If $\x=(1,3,3,2,5,5,5)\in\FR_7$, then its content is $(1,1,2,0,3,0,0)$.
\end{example}

Definition~\ref{def:content_fub} for the content  of a Fubini ranking $\textbf{a}$ is a weak composition of $n$ (with $n$ parts). Note that the converse is not true. 
For example, a weak composition starting with zero, would never arise as Fubini ranking.
In order to remove the entries with value zero we give the following definition.

\begin{definition}
    The \emph{reduced content} of $\x$ is the content of $\x$ with all entries of value $0$ removed.
\end{definition}

It is worth noting that the content can always be recovered from the reduced content because if entry $a_i\neq 0$ then the following $a_i-1$ entries will have value $0$. 
Unlike the content, the reduced content of two different Fubini rankings of length $n$ may have different lengths. For example, the Fubini ranking $(1,1,1,4,4)$ has reduced content $(3,2)$, while the Fubini ranking $(1,1,3,4,4)$ has reduced content $(2,1,2)$.
In light of these examples, we give the following definition.

\begin{lemma}\label{lem:valid reduced contents}
    Let $\textbf{a}= (a_1,a_2,\ldots,a_k)\in\mathbb{Z}_{\geq 1}^n$. Then there exists $\x$ a Fubini ranking of length $\sum_{i=1}^ka_i$ such that $\textbf{a}$ is the reduced content of $\x$.
\end{lemma}

\begin{proof}
    Let $\textbf{a}= (a_1,a_2,\ldots,a_k)\in\mathbb{Z}_{\geq 1}^n$. We can construct a Fubini ranking from this reduced content as follows:
    Insert $a_1$ many 1's. Then insert $a_2$ many $1+a_1$'s, and so on. 
    Thus
    \[\x=(\underbrace{1,\ldots,1}_{a_1\mbox{ times}},\underbrace{1+a_1,\cdots,1+a_1}_{a_2\mbox{ times}},\ldots,\underbrace{1+\sum_{i=1}^{j-1}a_i,\ldots,1+\sum_{i=1}^{j-1}a_i}_{a_j \mbox{ times}},\ldots,\underbrace{1+\sum_{i=1}^{k-1}a_i,\ldots,1+\sum_{i=1}^{k-1}a_i}_{a_{k}\mbox{ times}}).\]
    By construction this is a Fubini ranking, with content
    \begin{align}\label{almost home free}(a_1,\underbrace{0,\ldots,0}_{a_{1}-1}, a_2, \underbrace{0,\ldots,0}_{a_2-1 },\ldots, \underbrace{0,\ldots, 0}_{a_{j-1}-1 },a_j,\underbrace{0,\ldots,0}_{a_j-1},\ldots, \underbrace{0,\ldots,0}_{a_{k-1}-1} ,a_k, \underbrace{0,\ldots,0}_{a_k-1}).\end{align}
    The proof follows from the fact that the reduced content of the vector in \eqref{almost home free} is  $\textbf{a}$.
\end{proof}


Based on the construction of the previous result we have the following.
\begin{lemma}\label{lem:valid content}
    If $\textbf{a}= (a_1,a_2,\ldots,a_k)\in\mathbb{Z}_{\geq 1}^k$ is a composition of $n$, then there exists a weakly flattened Fubini ranking with runs of weak ascents having reduced content $\textbf{a}$.
\end{lemma}

\begin{proof}
    Consider a sequence with reduced content ${\bf a}$ in weakly increasing order constructed in \Cref{lem:valid reduced contents}. If we are allowing runs of weak ascents, this creates a single run. Thus this sequence will be a weakly flat Fubini ranking with a single run of weak ascents because, with only a single run, there is no way for the leading elements of runs to decrease.
\end{proof}

\begin{theorem}
\label{wwenum}
Let $\textbf{a}= (a_1,a_2,\ldots,a_k)\in\mathbb{Z}_{\geq 1}^n$ be a composition of $n$. Then the number of  weakly flattened Fubini rankings with runs of weak ascents that have reduced content $\textbf{a}$ is:
\footnotesize{$$\sum_{j_2=0}^{a_2}\left [\binom{a_1+a_2-j_2-2}{a_2-j_2}\sum_{j_3=0}^{a_3}\left [\binom{a_1+j_2+a_3-j_3-2}{a_3-j_3}\sum_{j_4=0}^{a_4}\left [\ldots \sum_{j_k=0}^{a_k}\left [\binom{a_1+a_k-j_k-2+\displaystyle\sum_{r=2}^{k-1}j_{r}}{a_k-j_k}\right ]\ldots\right ]\right ]\right ].$$}
\end{theorem}
\begin{proof} We proceed by constructing a weakly flattened Fubini ranking of length $n$ with runs of weak ascents. 
To do this we ensure that the reduced content of the Fubini ranking is indeed $\textbf{a}=(a_1,a_2,\ldots,a_k)$, as specified in the statement of the theorem. To do this, we first place the $a_1$ entries of value $1$ first. 
Since these are the first elements placed, this can be done in one way. Next we place the $a_2$ entries with value $a_1+1$ in such a way that $j_2$ of them come after the final $1$. Since we want $a_2 - j_2$ entries of value $a_1+1$ to come before the final $1$, but they must come after the first $1$ in order for the sequence to be flattened, there are $\binom{a_1+a_2-j_2-2}{a_2-j_2}$ ways to place the $a_2-j_2$ copies of the  value $a_1+1$ between the $1$s. 
Since the $a_2-j_2$ values of $a_1+1$ need to come after the first $1$ and before the last $1$, this case is only possible if $a_1\geq 2$. 
In the other case, where $a_1=1$, there is a unique way to place the $a_2$ values of $a_1+1$, immediately to the right of the $1$.

Continuing iteratively, if the reduced content contains three or more entries, we place the $a_3$ copies of the next value ($a_1+a_2+1$) so that $a_3-j_3$ of them come before the final entry of lesser value, which may be either $1$ or $a_1+1$. If an entry is placed between two entries of value $1$, then it must be placed to continue the run between those entries, otherwise the Fubini ranking will not be flattened. 
However, entries can be placed anywhere in between the $j_2$ entries with value $a_1+1$ that were placed after the final entry with value $1$. Thus there are $a_1+j_2-1$ places to put the new elements since they cannot go before the first $1$ or after the final element that has already been placed. Since there are $a_3-j_3$ values to be placed and multiple elements can be placed in between the same existing elements, this is equivalent to placing $a_3-j_3$ indistinguishable balls into $a_1+j_2-1$ urns. This can be done $\binom{a_1+j_2-2+a_3-j_3}{a_3-j_3}$ ways.

This continues until it is time to place the final $a_k$ entries with value $a_1+a_2+\cdots+a_{k-1}+1$. 
We will place $j_k$ of them at the very end of the sequence and the remaining $a_k-j_k$ values in the middle of the existing sequence. These $a_k-j_k$ entries can be placed anywhere between the $j_{k-1}$ entries of the next to largest value at the end of the sequence. 
However, if they are placed between entries with values less than the next to largest, then they must be placed so as to extend the existing run, otherwise the resulting sequence will not be flattened. Thus the values can be placed into the existing sequence in $\binom{a_1+j_2+j_3+\cdots+j_{k-1}+a_k-j_k-2}{a_k-j_k}$ ways.

Since at each step we are making independent choices, the number of ways to make them is the product of the number of choices at each step. However, we need to sum over all possible allowable values for the numbers $j_2,j_3,\ldots,j_k$.
This is exactly the stated formula in the result.
\end{proof}

By specifying the last entry of a Fubini ranking to be 1, we give the following result.

\begin{corollary}\label{cor:ends_in_1}
If $\textbf{a}= (a_1,a_2,\ldots,a_k)\in\mathbb{Z}_{\geq 1}^n$ is a composition of $n$, then
the number of weakly flattened Fubini rankings with runs of weak ascents that have reduced content $\textbf{a}$ ending in an entry with value 1 is:
    $$\prod_{i=2}^k\binom{a_1+a_i-2}{a_i}.$$
\end{corollary}

\begin{proof}
    As in the proof of \Cref{wwenum}, we construct Fubini rankings by first placing the $1$s, then the $a_1+1$s, and so forth. However, since every element larger than $1$ ends up between two elements of value $1$, they must be placed within existing runs in order to ensure that the Fubini ranking is flattened. 
    Thus, there are $a_1-1$ places between $1$s to place the $a_2$ entries of value $a_1+1$. 
    So the number of ways to construct the sequence at this step is given by $\binom{a_1+a_2-2}{a_2}$.
    For each such sequence,
    we now place the $a_3$ values of $a_1+a_2+1$ into the same $a_1-1$ gaps between instances of $1$s and once we choose which gaps to place those values  in, the Fubini ranking we construct is unique since those values must be placed at the end of the run.
    This will ensure the constructed Fubini ranking is flattened.
    The number of ways to do this is given by $\binom{a_1+a_3-2}{a_3}$. Since these are independent choices, at this step the current total count of sequences constructed is $\binom{a_1+a_2-2}{a_2}\binom{a_1+a_3-2}{a_3}$.
    Continuing this process for each remaining $4\leq i\leq k$, the independence of the choices implies that the final number of weakly flattened Fubini rankings with runs of weak ascents ending with a value of $1$ is given by the product 
    \[\prod_{i=2}^k\binom{a_1+a_i-2}{a_i}.\qedhere\]
\end{proof}

We conclude by noting the following implications of \Cref{cor:ends_in_1}. First, whenever the tuple $\textbf{a}=(1,1,\ldots,1)\in\mathbb{Z}^n$, the number of weakly flattened Fubini rankings with runs of weak ascents ending with a value of $1$ is zero. As should be expected since having $\textbf{a}=(1,1,\ldots,1)\in\mathbb{Z}^n$ means the Fubini ranking is a permutation, and a permutation ending in $1$ (provide $n>1$) cannot be flattened.
Second, the count given in the result is constant when permuting the $k-1$ last entries in $\textbf{a}$.  

\subsection{Weakly flattened Fubini rankings with \texorpdfstring{$k$}{k}~runs of weak ascents}
\Cref{tab:weakly flattened Fubini with k runs} provides data on the number of weakly flattened Fubini rankings with $n$ competitors and $k$ runs of weak ascents, which we denote by $\wflatwruns_k(\FR_n)$. We remark that it is an open problem to determine formulas for $\wflatwruns_k(\FR_n)$. 
However, the row sums are given by the formula in \Cref{wwenum}.

%


\begin{table}[h!]
\centering
\begin{tabular}{|c||c|c|c|c|c|}
\hline
$n\backslash k$ & 1 & 2 & 3 & 4 & 5 \\
\hline\hline
1 & 1 & 0 & 0 & 0 & 0 \\
\hline
2 & 2 & 0 & 0 & 0 & 0 \\
\hline
3 & 4 & 2 & 0 & 0 & 0 \\
\hline
4 & 8 & 16 & 0 & 0 & 0 \\
\hline
5 & 16 & 84 & 16 & 0 & 0 \\
\hline
6 & 32 & 368 & 244 & 0 & 0 \\
\hline
7 & 64 & 1464 & 2264 & 208 & 0 \\
\hline
8 & 128 & 5504 & 16632 & 5080 & 0 \\
\hline
9 & 256 & 19984 & 106808 & 72504 & 3776 \\
\hline
10 & 512 & 70976 & 630016 & 794552 & 133792 \\
\hline
\end{tabular}
\caption{Values for $\wflatwruns_k(\FR_n)$, the number weakly flattened Fubini rankings on $n$ competitors with $k$ runs of weak ascents.}\label{tab:weakly flattened Fubini with k runs}
\end{table}

We conclude with the following set of results related to the value of $\wflatwruns_k(\FR_n)$. Note that \Cref{runOne}, \Cref{tooManyRuns}, and \Cref{thm:wfFub} assume that the runs in the Fubini ranking are defined with weak ascents and that the Fubini ranking be at least weakly flattened, but all of the proofs will hold if the condition is strengthened to flattened.

\begin{lemma}\label{runOne}
    In a flattened Fubini ranking with runs of weak ascents the only run which can contain only a single element is the final run in the sequence. 
    Likewise, in a weakly flattened Fubini ranking with runs of weak ascents the only run which can contain only a single element is the final run in the sequence.
\end{lemma}

\begin{proof}
    Suppose the first run in a flattened Fubini ranking with runs of weak ascents contains only one element. Since every Fubini ranking starts with a leading element of value $1$, if there is a second element in the Fubini ranking then it would continue the run. Thus the whole Fubini ranking must contain a single element and the run of that single element is the final run in the sequence.

    Suppose, for contradiction, that there is a run containing a single element, which is followed by another run in the Fubini ranking. Since the Fubini ranking consists of weak ascents, we can conclude that the element of the latter run is strictly less than the single element in the former run. 
    Namely, there is a descent in the Fubini ranking from the end of a run to the start of the new run.
    However, this contradicts the assumption that the Fubini ranking is flat, which requires that the leading elements of runs (weakly or strictly) increase from left to right.

    Thus any run which contains a single element must not be followed by another run. Therefore, it must be the final run occurring in the sequence.
\end{proof}

\begin{corollary}
\label{tooManyRuns}
    Let $n\geq 1$. There exists a flattened Fubini ranking of length $n$ with runs of weak ascents. Likewise, there exists a weakly flattened Fubini ranking of length $n$ with runs of weak ascents.  
\end{corollary}
\begin{proof}
Clearly if $n=1$ the entire Fubini ranking is a run with a single entry. Assuming $n>1$ the following arguments show both directions of the corollary statement.

    ($\Rightarrow$) By \Cref{runOne} any run with a single entry is the final run of the flattened (or weakly flattened) Fubini ranking, so there has to be at least two numbers. However, if there are exactly two numbers then the Fubini ranking would consist of a descent, which would imply that it is not flattened.  

($\Leftarrow$)    For $n>2$ it is possible to construct a Fubini ranking that contains a run with a single entry in the following way. Consider the permutation $1,3,4,\ldots,n,2$. Since it is a permutation, it is a Fubini ranking, and it will be both flattened and weakly flattened as they are equivalent if there are not two elements of the same value. 
Clearly it has two runs where the first begins with the element $1$ and the second with element $2$, so it satisfies being weakly flattened, and contains a run with a single element.
\end{proof}


As illustrated by \Cref{tooManyRuns}, in order to be flattened a Fubini ranking cannot have too many runs. This leads us to the following result.


\begin{theorem}\label{thm:wfFub}
Let \emph{$|\flatwruns_k(\FR_n)|$} denote the number of flattened Fubini rankings, and \emph{$|\wflatwruns_k(\FR_n)|$} denote the number of weakly flattened Fubini rankings, each with runs of weak ascents, of length $n>0$, with $k>0$ runs of weak ascents. Then \emph{$|\wflatwruns_k(\FR_{j})|=0$}, exactly when $k>\lceil \frac{j}{2}\rceil$ for any $j\in \mathbb{Z}_{\geq 1}$. Likewise, \emph{$|\flatwruns_k(\FR_{j})|=0$} exactly when $k>\lceil \frac{j}{2}\rceil$ for any $j\in\mathbb{Z}_{\geq 1}$.
\end{theorem}

\begin{proof}
From Lemma~\ref{runOne} we know that there is at most one run in the Fubini ranking which contains a single element. 
So, since there can be one run of length one and all of the remaining runs require at least two elements, we have that the most runs that could exist in a Fubini ranking of length $j$ is $\lfloor \frac{j-1}{2}\rfloor+1$.

When $j$ is even, $\lceil\frac{j}{2}\rceil = \frac{j}{2}$ will be one larger than $\lfloor\frac{j-1}{2}\rfloor$ so $\lfloor\frac{j-1}{2}\rfloor + 1 = \lceil\frac{j}{2}\rceil$. 
Similarly, if $j$ is odd, then $\lceil\frac{j}{2}\rceil$ is one larger than $\frac{j-1}{2}=\lfloor\frac{j-1}{2}\rfloor$ so, once again, $\lfloor\frac{j-1}{2}\rfloor + 1 = \lceil\frac{j}{2}\rceil$. 
This shows that the maximum number of runs which can be contained in a weakly flattened Fubini ranking of length $j$ is $\lceil\frac{j}{2}\rceil$.

To show that $|\wflatwruns_k(\FR_{j})|>0$ when $0<k\leq \lceil\frac{j}{2}\rceil$ it suffices to show we can construct a Fubini ranking of the specified length with the requisite number of runs. 
To achieve the upper-bound of $\lceil\frac{j}{2}\rceil$ consider a Fubini ranking where all entries are unique.
Namely, the Fubini ranking will be a permutation of $[n]$ and corresponds to a ranking with no ties. 
Now we arrange the elements as follows: $1,j,2,j-2,\ldots,\lfloor\frac{j+1}{2}\rfloor$. 
If $j$ is even, this will result in $\frac{j}{2}=\lceil\frac{j}{2}\rceil$ different runs, each containing two elements. 
If $j$ is odd, there will be $\frac{j-1}{2}$ runs with two elements and a final run containing only a single element, which gives a total of $\frac{j-1}{2}+1 = \lceil\frac{j}{2}\rceil$ runs.

If you wish to reduce the total number of runs in the Fubini ranking, simply take two adjacent runs then sort their elements in increasing order to create a single run. 
The smallest element of the resulting run is unchanged and the largest element has increased.
This merges the two adjacent runs without affecting the preceding of the subsequent runs, if they exist. 
In this way, a Fubini ranking with $k$ runs can be obtained for any value of $0<k\leq\lceil\frac{j}{2}\rceil$.
\end{proof}

\subsection{Weakly flattened Fubini rankings with runs of weak ascents and a fixed content}
In this section, we specify the content of a weak flattened Fubini ranking with runs of weak ascents and provide a result on the maximum number of runs of such a Fubini ranking. 
Enumerative results for these sets are difficult and so we begin by further restricting and giving a formula for a subset having a fixed final entry of value $1$.

\begin{proposition}\label{prop:end_in_1}
Let 
$\textbf{a}=(a_1,a_2,\ldots,a_n)$ denote the content of a Fubini ranking with $k\leq a_1\neq n-k+1$. Let
$F_{n,k}(\bf{a})$ be the number of weakly flattened Fubini rankings of length $n$ with content $\bf{a}$ which have $k$ runs of weak ascents and also end in the value $1$. Then 
    \[F_{n,k}(\textbf{a})=\binom{a_1-1}{k-1}\cdot|\mathcal{B}_k(\textbf{a})|,\]
    where $\mathcal{B}_k(\bf{a})$ is the set of 
$n\times (k-1)$ nonnegative integer valued matrices whose $i$th row is a weak composition of $a_i$ and whose column sums are always greater than or equal to $1$.
\end{proposition}
\begin{proof}
    In this case, as we must end with a 1, every run must start with a 1, if the Fubini ranking is to be flattened. 
    Hence, there can be no more runs than number of ones in the Fubini ranking.
    Let $\textbf{a}=(a_1,a_2,\ldots,a_n)$ be the content of the Fubini ranking. 
    To count  weakly flattened Fubini rankings with $k$ runs of weak ascents that end with a 1, we observe that the runs will always begin with a consecutive list of ones, and as there are $k$ runs, we need to have at least $k-1$ elements that are not ones, to place after the sequence of ones. 
    Thus, we must have that the number of ones is at least $k$ and no more than $n-k+1$. 
    Namely, 
    $k\leq a_1\leq n-k+1$. 
This now becomes a balls in urns problem, where we first create a composition of $a_1$ into $k$ parts, to account for the start of the distinct $k$ runs, all of which must begin with a 1.
Then for each such composition $c=(c_1,c_2,\ldots,c_k)\vdash a_1$, with $c_i>0$, we start each run with exactly $c_i$ ones. 
Note that there are $\binom{a_1-1}{k-1}$ such compositions.
Next, we must place the remaining $n-a_1$ elements ensuring to always place at least one of these elements between each of the runs. 
So we must account for all possible ways in which the elements greater than 1 appear in those bins ensuring that no bin is ever left empty, as this would reduce the number of runs. 
Note that once we know what elements appear in a given run/bin, we simply order the values in weakly increasing fashion. 
So first, for each $i=2,3,\ldots,n$ we count the number of ways in which the value $i$, of which there are $a_i$ copies, get distributed among the $k-1$ bins. 
In general, this can be done by counting weak compositions with $k-1$ parts. 
However, some care is required as when we distribute all values $i=2,3,\ldots,n$ among the $k-1$ bins, we must ensure every bin is nonempty.
The number of ways to place the elements into those bins is given by the number of 
$n\times (k-1)$ nonnegative integer valued matrices whose $i$th row is a weak composition of $a_i$ and whose column sums are always greater than or equal to $1$.
Let $\mathcal{B}_k(\bf{a})$ denote all such matrices. 
Thus the number of flattened Fubini rankings with weak runs which end in a value of 1 and have content $\textbf{a}=(a_1,a_2,\ldots,a_n)$ is given by 
\[\binom{a_1-1}{k-1}\cdot|\mathcal{B}_k(\bf{a})|.\qedhere\]
\end{proof}

It would be of interest to find a formula for the number of matrices in $\mathcal{B}(\textbf{a})$ for any $\textbf{a}$ content of a Fubini ranking.
We list this as an open problem in the future work section. 
Next, we return to counting the maximum number of runs appearing in a flattened Fubini ranking with runs of weak ascents.

\begin{theorem}\label{thm:sec6}
    Let ${\bf a} = (a_1,a_2,\ldots,a_n)$ be a weak composition of $n$ and the content of a Fubini ranking with $n$ competitors. Let $f$ be a Fubini ranking with content $\textbf{a}$.
     Let $a$ be the median entry in $f$, $m_a$ be the number of entries in $f$ strictly less than the median $a$, and $M_a$ be the number of entries in $f$ strictly greater than the median $a$. 
    Then exactly one of the following holds.
    \begin{enumerate}
        \item If $m_a = M_a = 0$, then the maximum possible number of runs in $f$ is $1$.\label{s1}
        \item If $0=m_a<M_a$, then the maximum possible number of runs in $f$ is $M_a+1$.\label{s2}
        \item If $0=M_a<m_a$, then the maximum possible number of runs in $f$ is $m_a$.\label{s3}
        \item If $m_a,M_a>0$, then the maximum possible number of runs in $f$ is the minimum of $m_a + M_a$ and $\lceil\frac{n}{2}\rceil$.\label{s4}
    \end{enumerate}
\end{theorem}

\begin{proof}
    Since the runs of $f$ are runs of weak ascents, every run of $f$, other than the first run, begins with an element that is strictly less than the element before it. Because the leading elements of runs must be weakly increasing, this means the element before the leading term of any run cannot be the leading term of its own run. 
    Thus, each run in $f$, except for the last run, must contain at least two elements. If $n$ is odd, then the last run of $f$ could have length $1$.
 Therefore, the number of runs in $f$ cannot be greater than $\lceil \frac{n}{2}\rceil$.
 
 In what follows let $f^\uparrow=(f_1',f_2',\ldots,f_n')$ be the weakly increasing rearrangement of the values in $f$, so that $f_i'\leq f_{i+1}'$ for all $1\leq i\leq n-1$.

For statement \eqref{s1}:
   If $n=2x+1$ is odd, then the median of $f$ is $f_{x+1}$ as the median of the set of numbers in $f$ is the middle element when the values are arranged in weakly increasing order. If $n=2x$ is even, the the median of $f$ is the average (arithmetic mean) of $f_{x}'$ and $f_{x+1}'$, which are the two middle elements in $f'$. 
   If the number of elements greater than the median and the number of elements less than the median are both zero, i.e., $m_a=M_a=0$, then all the elements in $f'$, and hence in $f$, must be equal to the median value. In this case, $f$ has exactly one run, which completes the proof of statement \eqref{s1}.

    For statement \eqref{s2}:
    If $m_a=0$, then there are no elements less than the median in $f$. If there are $M_a>a$ elements greater than the median, then the first $\lceil\frac{n}{2}\rceil$ elements in $f'$ must be the same. Therefore, one can construct $f$ so as to have $M_a$ runs by beginning each run with a copy of the median $a$ followed by a single element greater than the median, of which there are $M_a$. However, since the minimal element in $f$ is also the median element, there will be more instances of $a$ in $f$ than the $M_a$ elements greater than $a$. Hence we can construct a final run in $f$ by placing all remaining copies of the value $a$ after the final element greater than it. This construction yields a total of $M_a+1$ runs. 
Note, to begin a run, aside from the first one, there must be a descent. Given that we have $M_a$ numbers larger than $a$, any Fubini ranking with this content, can have at most $M_a$ many descent. Hence, the maximum number of runs is $M_a+1$, achieved via the construction described.

    For statement \eqref{s3}: Assume $0=M_a<m_a$. If there are no elements greater than the median $a$, and there are $m_a$ elements less than the median $a$, we may only construct at most $m_a$ runs using the numbers smaller than $a$ to start new runs. To construct $m_a$ number of runs, let each of the $m_a$ elements less than the median begin its own run followed by a single element with the median value, thereby creating runs of length two. 
There will be some remaining copies of the median element in $f$, but those values can not begin any new runs because they will have the same value as the largest element in each of the existing runs. Thus we can construct at most $m_a$ runs.

    For statement \eqref{s4}: Assume $m_a,M_a>0$.
Hence, there are elements in $f$ that are less than the median value $a$ and elements greater than the median value $a$. 

We break this proof into two cases:
\begin{enumerate}
    \item if $m_a+M_a> \frac{n}{2}$, or
    \item if $m_a+M_a \leq \frac{n}{2}$.
\end{enumerate} 
In Case (1): Assume that $m_a+M_a>\frac{n}{2}$. We will construct runs of length two in the following way. 
First, list all of the elements in weakly increasing order then pair the element at position $i$ in the list with the element at position $\lceil\frac{n}{2}\rceil+i$ for $1\leq i\leq \lfloor\frac{n}{2}\rfloor$. 
{This creates pairs of elements that are strictly increasing, because each pair contains an element two the left of the middle value and an element to the right of the middle value, so if the values are equal, then they must both be equal to the median value. 
However, they are $\lceil\frac{n}{2}\rceil$ entries apart and, by assumption, $m_a+M_a>\frac{n}{2}$ implies that there are fewer than $\frac{n}{2}$ elements with the median value. Similarly, the second element in each pair is greater than the next element in the sequence because they are $\lceil\frac{n}{2}\rceil-1$ entries apart, but if they had the same value there would need to be $\lceil\frac{n}{2}\rceil$ entries of the same value, which would force that value to be the median entry and we know that there are fewer than $\frac{n}{2}$ entries with the median value.}

Since each pair is strictly increasing and the second element of the pair is strictly greater than the next element, we will have created $\lfloor\frac{n}{2}\rfloor$ runs of length two with a final run of length one if $n$ is odd. Thus there are $\lceil\frac{n}{2}\rceil$ total runs. Since each run has length one or two, by \Cref{runOne} we know this is the most runs possible for a Fubini ranking of this length.

In Case (2):
In the case where $m_a+M_a\leq\frac{n}{2}$, we know that at least half of the entries in the eventual Fubini ranking will have the median value. Specifically, if $n_a$ is the number of times the element of median value appears, $m_a+M_a \leq n_a$. 
In weakly increasing order we pair the first $m_a-1$ elements less than the median with an element with median value to create $m_a-1$ runs. 
Next, we pair the largest element less than the median with the largest element greater than the median, which must exist since $M_a>0$. This creates $m_a$ runs consisting of two elements. Next, we form $M_a-1$ pairs that start with an element of the median value and finish with one of the remaining elements greater than the median, we place these at the end of our constructed sequence in weakly increasing order. 
At this point we have constructed a total of $m_a+M_a-1$ runs, all but one of which contains an element of median value. However, since we know that $m_a+M_a-1 < n_a$, there will be elements of the median value that remain unused and we must place them somewhere. 
We place them at the end of the Fubini ranking to form a final run and hence we obtain a total of $m_a+M_a$ runs, as desired.

To show that this is the maximum possible number of runs, consider for contradiction that, if we have $m_a+M_a+1$ runs then either there are multiple runs which contain only the median element, or there is one run that contains only the median element and every other run must contain at most one element that is not the median element. If there are multiple runs that consist exclusively of the median element, then they must be adjacent to each other or there is at least one run that commences with the median element between them, because the ranking is weakly flattened. However, since the runs consist of weak ascents, if the values are adjacent, then they are a single run, and if there are runs beginning with the median element between them, the first such run is part of the leading run of only the median element by a similar argument. Thus, a weakly flattened Fubini ranking with runs of weak ascents cannot contain multiple runs consisting only of entries of a single value.

With that in mind, $m_a+M_a$ runs are required to exist. Moreover, these runs contain the median element at least once  and exactly one non-median, which then allows us to obtain $m_a+M_a+1$ runs in a ranking with this content. 
If such a run contains one of the $m_a$ elements less than the median element, it must be the first element in the run. 
If such a run contains one of the $M_a$ elements greater than the median element, it must be the last such element in the run. 
Therefore, the runs containing elements less than the median must always come before runs containing elements greater than the median. Since $m_a,M_a>0$, both types of runs will exist, and let us consider what happens between the last run containing an element less than the median and the first run containing an element greater than the median. 
If there is nothing else between them in the final constructed sequence, they will merge into a single run, because the run containing the element less than the median ends in the median element, the run containing the element greater than the median begins with the median element, and we are considering runs of weak ascents. If there is something between them, it must be a run consisting of only the median element, and then all three runs will merge into a single run by a similar argument. Therefore, it is impossible to create a weakly flattened Fubini ranking with runs of strict ascents with the specified content that contains more than $m_a+M_a$ runs. 
Since we can construct one that does contain $m_a+M_a$ runs, this is the maximum number of runs possible in such a ranking.

By Cases (1) and (2): When $m_a+M_a>\frac{n}{2}$ we know that $m_a+M_a\geq \lceil\frac{n}{2}\rceil$ and we can construct a weakly flattened Fubini ranking with runs of weak ascents that contains $\lceil\frac{n}{2}\rceil$ runs and this is the maximum number of runs possible. 
When $m_a+M_a\leq \frac{n}{2}$ we know that $m_a+M_a\leq \lceil\frac{n}{2}\rceil$ and we can construct a weakly flattened Fubini ranking with runs of weak ascents that contains $m_a+M_a$ runs and this is the maximum number of runs possible. 
Thus, in either case, we can produce a Fubini ranking containing $\min(m_a+M_a,\lceil\frac{n}{2}\rceil)$ different runs and that is the most that such a ranking could possibly contain, as desired.
\end{proof}

\section{Flattened Fubini rankings with runs of ascents}\label{sec:strongflat}

In this section, we consider the set of Fubini rankings of length $n$ with runs of strict ascents and whose leading terms are in increasing order. 
We denote this subset of Fubini rankings by $\flatruns(\FR_n)$.

Before stating and proving our next result we illustrate the proof with the following example.
\begin{example}
We can construct a flattened Fubini ranking with runs of strict ascents with reduced content $\textbf{a}=(1,2,3)$ as follows. 
First place the 1 to start the Fubini ranking. 
Then we have two 2's. and three 4's to place. As we want the leading terms to be in increasing order we now place a two and a 4 following the 1.
This gives the partial Fubini ranking: $(1,2,4)$. 
We now have remaining one 2 and two 4's. 
To start a new run we place a 2 and a 4, giving the partial Fubini ranking $(1,2,4,2,4)$. 
We now have one remaining 4, which we place at the end constructing the final Fubini ranking $(1,2,4,2,4,4)$. Note that the runs are 124,24,4. The leading terms are in increasing order and the suns consist of strict ascents. Thus we have created a Fubini ranking with the correct properties. 
\end{example}




\begin{theorem}
\label{sscontent}
Let ${\bf a} = (a_1,a_2,\ldots,a_k)\in \mathbb{Z}^k_{\geq 1}$ be a composition of $n$ with $k$ parts. Then there exists a flattened Fubini ranking with runs of ascents whose reduced content is $\textbf{a}$ if and only if $a_{i}\leq i$ for all $1\leq i\leq k$.
\end{theorem}

\begin{proof}
    ($\Leftarrow$) This direction of the proof is constructive. Given a composition ${\bf a}=(a_1,a_2,\ldots,a_k)$ of $n$ with $k$ parts where the $i$th entry is at most $i$, we construct a flattened Fubini ranking with runs of strict ascents as follows.
    First let $v_i$ denote the value corresponding to entry $a_i$ in the reduced content $\textbf{a}$. 
    As $\textbf{a}$ is reduced, there are no zeros in $\textbf{a}$. 
    Moreover, note that for each $i\in[k]$, there are $a_i$ copies of the value $v_i$ in the Fubini ranking we construct. 
    
    To begin start the sequence with the leading $v_1=1$, because this is the smallest element possible, there will only be one of them as otherwise multiple runs would begin with the same element and the ranking would not be flattened. 
    So $a_{1}\leq 1$ and, since this is a Fubini ranking, there must be at least one $1$. Thus $a_1=1$.

At this point we have constructed the word $w_1=1$. 
Next we know $a_2\leq 2$, so there are at most two copies of the value $v_2$, which in this case must be equal to 2.
We place those 2's at the end of the word $w_1$ to construct the word $w_2=12$ if $a_2=1$ or $w_2=122$ if $a_2=2$. In either case the constructed word is flattened with runs of ascents. 

We continue iteratively in this fashion and consider trying to place $a_{i}\leq i$ elements of value $v_i$ in the word $w_{i-1}$. 
    By assumption, we are placing the values $v_i$, which are the $i$th largest value present in the Fubini ranking, and we have $a_i$ copies of that value. 
    If $a_i\leq 2$, then we place $a_i$ copies of the value $v_i$ at the far right of the word $w_{i-1}$ and call this $w_i$. This constructs a flattened Fubini ranking with  runs of strict ascents as the first copy of the value $v_i$ continues the last run in $w_{i-1}$, and if there is a second copy of $v_i$, it would start a new run in $w_i$, which could have  a leading term larger than any of the previous runs in $w_{i-1}$. 
    Now if $3\leq a_i\leq i$, once again place two copies of $v_i$ at the end of $w_{i-1}$, as in the previous case. 
    Now given that the content $\textbf{a}$ is reduced, there are $i-1$ distinct values smaller than $v_i$ in the word $w_{i-1}$. These are the values $v_1,v_2,\ldots,v_{i-1}$ of which we have already placed in $w_{i-1}$. 
    At this point, we still have $a_i-2$ copies of the value $v_i$. 
    We now place a copy of $v_i$ to the left of the rightmost value $v_j$ for $1<j\leq a_{i}-1\leq i-1$. This ensures that we do not place a copy of $v_i$ to the left of the leading $1$ in $w_{i-1}$. 
    Moreover, by placing the $a_i-2$ remaining copies of the values $v_i$
    in this way, ensures the following. 1. The value $v_i$ continue a previous runs in $w_{i-1}$, and 2. the rightmost value $v_j$ begins a new run, which continues to satisfy the flatness condition because, by construction, the rightmost element of value $v_j$ was placed to the right of all elements with values less than $v_j$. That is, there is no entry in $w_{i-1}$ to the right of the rightmost element of value $v_j$ which is smaller than $v_j$ with which to start a new run.  
    
    The result of the insertion of the $a_i$ copies of the value $v_i$ yields the word $w_i$, which is a flattened Fubini ranking with runs of  ascents.
    Continuing this process for all $i\in[k]$ culminates in the Fubini ranking with the desired properties.





($\Rightarrow$)   To show that the $i$th non-zero element of the reduced content cannot occur more than $i$ times, consider that each of the elements of that value must occur in a separate run, since runs are strictly increasing. The runs must have distinct leading elements, since we require the Fubini ranking to be flattened, and the leading elements must be less than or equal to the specified element. Thus, there are only $i$ possible leading elements, the $i$ elements that show up in the Fubini ranking that are less than or equal to the given element's value.

    This shows that a proposed reduced content ${\bf a}=(a_1,a_2,\ldots,a_k)$ will be the reduced content of some flattened Fubini ranking with runs of ascents if and only if $a_i\leq i$ for all $i\in[k]$.
\end{proof}

The proof of Theorem~\ref{sscontent} relies heavily on the iterative construction presented. It is not always possible to start with a Fubini ranking with reduced content $(a_1,a_2,\ldots,a_k)$ that is flattened with runs of ascents and then use that to build a flattened Fubini ranking with runs of ascents with reduced content $(a_1,a_2,\ldots,a_k,k+1)$. 
For example, the Fubini ranking $(1,3,2)$ has reduced content $(1,1,1)$ and is flattened with runs of ascents. However, we cannot use that Fubini ranking to construct a flattened Fubini ranking with runs of ascents with reduced content $(1,1,1,4)$. This is because we would need to
insert four elements of value $4$ into $(1,3,2)$ and any way of doing this either violates being flattened or having runs of ascents.

We remark that the example in the preceding paragraph is not a counterexample to Theorem~\ref{sscontent}, because we can find a flattened Fubini ranking with runs of ascents which has the reduced content $(1,1,1,4)$. Consider $(1,4,2,4,3,4,4)$ which has content $(1,1,1,4)$ and whose leading elements of runs ascents appear in order $1$, $2$, $3$, and $4$. Hence, $(1,4,2,4,3,4,4)$ is flattened with runs of ascents.

From computational experiments we present the following conjecture and welcome a proof.

\begin{conjecture}
\label{genfunc} 
Let 
$s_n=|\flatruns(\FR_n)|$ denote the number
of flattened Fubini rankings of length $n$ with runs of ascents. Then 
the generating function for $s_n$, is given by 
\[\sum_{n\geq 0}s_nx^n=\sum_{n\geq 0} \frac{1} { 2^{n+1}} \left( \prod_{i=1}^{n}( 1 + x(1+x)^i )\right).\]
\end{conjecture}
\Cref{genfunc} implies that 
    the number of flattened Fubini rankings with runs of ascents appears to agree with \cite[\seqnum{A338793}]{OEIS},
for which no combinatorial description is provided. 
Thus, a proof to \Cref{genfunc} would yield a combinatorial interpretation for this sequence.

\subsection{Flattened Fubini rankings with \texorpdfstring{$k$}{k}~runs of ascents}
\Cref{tab:strongly flattened Fubini with k runs} provides data on the number of flattened Fubini rankings with $n$ competitors and $k$ runs of ascents. It is an open problem to determine formulas for $|\flatruns_k(\FR_n)|$ in general. We do provide the following conjecture in the case where $n$ is odd, and $k$ is as large as possible relative to $n$.

\begin{table}[h]
\centering
\begin{tabular}{|c||c|c|c|c|c|}
\hline
$n\backslash k$ & 1 & 2 & 3 & 4 & 5 \\
\hline\hline
1 & 1 & 0 & 0 & 0 & 0 \\
\hline
2 & 1 & 0 & 0 & 0 & 0 \\
\hline
3 & 1 & 2 & 0 & 0 & 0 \\
\hline
4 & 1 & 8 & 0 & 0 & 0 \\
\hline
5 & 1 & 24 & 10 & 0 & 0 \\
\hline
6 & 1 & 64 & 92 & 0 & 0 \\
\hline
7 & 1 & 162 & 554 & 82 & 0 \\
\hline
8 & 1 & 400 & 2772 & 1352 & 0 \\
\hline
9 & 1 & 976 & 12560 & 13656 & 938 \\
\hline
10 & 1 & 2368 & 53684 & 109672 & 24236 \\
\hline
\end{tabular}
\caption{Flattened Fubini rankings on $n$ competitors with $k$ with runs of ascents.}\label{tab:strongly flattened Fubini with k runs}
\end{table}

\begin{conjecture}\label{conj:2}
        Let $s(n,k)=|\flatruns_k(\FR_n)|$ be the number of flattened Fubini rankings on $n$ competitors with $k$ runs of ascents. Then
\begin{align}\label{eq:strong flat with nk}
s(2j+1,j-1)=\sum_{i=0}^{j-1} E_2(2j+1,i)2^{i},\end{align}
    where $E_2(a, b)$ are the second-order Eulerian numbers and are OEIS sequence \cite[\seqnum{A340556}]{OEIS}. 
\end{conjecture}
The sequence $(s(2j+1,j-1))_{j\geq 1}$ 
defined in \Cref{eq:strong flat with nk}
agrees with OEIS sequence \cite[\seqnum{A112487}]{OEIS}.
We recall that the second-order Eulerian number $E_2(n, k)$ is the number of Stirling permutations of order $n$ with exactly $k$ descents, and the Stirling permutations of order $n$ are permutations of the multiset $\{1,1,2,2,3,3,\ldots,n,n\}$ such that the numbers appearing between any pair of values $i$, must be strictly larger than $i$.
Flattened Stirling permutations were studied in \cite{flat_sp} and shown to be in bijection with type $B$ set partitions. 
It is an open problem to prove \Cref{conj:2} and to establish formulas for $s(n,k)$ for general values of $n,k\geq 1$.


\section{Mixed flattened Fubini rankings}\label{sec:mixed_conditions}
In this section, we consider the sets 
$\wflatruns(\FR_n)$, the set of weakly flattened Fubini rankings with runs of ascents, and $\flatwruns(\FR_n)$, the set of flattened Fubini rankings with runs of weak ascents.

\subsection{Weakly flattened with runs of ascents}
We begin with the following result which is evocative of Theorem~\ref{wwenum}, where we counted weakly flattened Fubini rankings with runs of weak ascents and a specified reduced content.

\begin{theorem}
\label{wsenum}
    Let $(a_1,a_2,\ldots,a_k)$, a composition of $n$ containing positive entries, be the reduced content of a Fubini ranking. 
    Then the number of weakly flattened Fubini rankings with runs of ascents having reduced content $(a_1,a_2,\ldots,a_k)$ is given by \footnotesize{$$\sum_{j_2=0}^{a_2}\left (\binom{a_1-1}{a_2-j_2}\sum_{j_3=0}^{a_3}\left (\binom{a_1+j_2-1}{a_3-j_3}\sum_{j_4=0}^{a_4}\left (\ldots \sum_{j_k=0}^{a_k}\left (\binom{a_1+j_2+j_3+\cdots+j_{k-1}-1}{a_k-j_k}\right )\ldots\right )\right )\right ).$$}
\end{theorem}

\begin{proof}
    As in the proof of Theorem~\ref{wwenum}, we construct a weakly flattened Fubini ranking by placing the $a_1$ entries of value $1$ first (and later the larger values). To start we place all of the $a_1$ values of $1$, which can only be done one way. 
    Notice that since we are considering runs of ascents, each $1$ begins its own unique run and, since we are considering weakly flattened Fubini rankings, the sequence of $1$s will satisfy the weak flat conditions.
    
    Next, we place the $a_2$ entries with value $a_1+1$ in such a way that $j_2$ of them come after the final~$1$. Since we want $a_2 - j_2$ entries of value $a_1+1$ to come before the final $1$, but in order for the sequence to be flattened they must come after the first $1$. Furthermore, if we place two elements of value $a_1+1$ between the same two $1$s, the second will start its own run, which we cannot allow. Thus, at most one element of value $a_1+1$ may be allowed between any two consecutive $1$s. Since we have $a_1$ elements of value $1$, there are $a_1-1$ positions where we might place an element of value $a_1+1$ between them. Thus, there are $\binom{a_1-1}{a_2-j_2}$ ways to place the $a_2-j_2$ elements of value $a_1+1$ among the $1$s.

Continuing iteratively, if the reduced content contains three or more entries, we place the $a_3$ copies of the next value $(a_1+a_2+1)$ so that $a_3-j_3$ of them come before the final entry of lesser value (this way we are building longer runs and not new runs). If an entry is placed between two entries of value $1$, then it must be placed to continue the run between those entries, otherwise the Fubini ranking will not be flattened. Furthermore, only one element of the next value may be placed between any two elements of lesser value, since the second value placed would begin a new run (as the runs are of ascents and any value repeated twice appearing consecutively would begin a new run). However, entries can be placed anywhere between the $j_2$ entries with value $a_1+1$ that were placed after the final entry with value $1$. This leaves a total number of $a_1+j_2-1$ positions in which the $a_3-j_3$ elements can be placed, which can be done $\binom{a_1+j_2-1}{a_3-j_3}$ in different ways.

This continues until it is time to place the final $a_k$ entries with the value $a_1+a_2+\ldots +a_{k-1}+1$. We place $j_k$ of them at the very end of the sequence and $a_k-j_k$ in the middle of the existing sequence. These $a_k-j_k$ entries can be placed anywhere between the $j_{k-1}$ entries of the next to largest value at the end of the sequence as long as only one is placed between two consecutive elements. However, if they are placed between the $a_1$ $1$s or the $j_2$ entries with value $a_1+1$ after the $1$s or the $j_3$ entries of the third value after the $a_1+1$s, or so forth, then they must be placed so as to extend the existing run and, again, only one such element can be added to each run. Thus, the values can be placed into the existing sequence in $\binom{a_1+j_2+j_3+\ldots+j_{k-1}-1}{a_k-j_k}$ ways.

Since at each step we are making independent choices, the number of ways to construct these Fubini rankings is the product of the number of choices at each step. However, we need to sum over all possible allowable values for $j_2, j_3,\ldots j_k$. This is exactly the stated formula in the result.
\end{proof}

The following two results give bounds on the number of runs appearing in a weakly flattened Fubini ranking with runs of ascents and a specified reduced content. 

\begin{lemma}
Consider the subset of \emph{$\wflatruns_k(\FR_n;\textbf{a})\subseteq\wflatruns_k(\FR_n)$} of weakly flattened Fubini rankings with $n$ competitors, $k$ runs of ascents, and content ${\bf a} = (a_1,a_2,\ldots,a_n)$. 
If the set \emph{$\wflatruns_k(\FR_n;\textbf{a})$} is nonempty, then $\max\emph{\textbf{a}}\leq k$.
\end{lemma}

\begin{proof}
    Since each run is made up of consecutive ascents, no two elements with the same value can be in the same run. 
    This immediately makes the maximal entry in the content a lower bound for the number of runs the Fubini ranking can contain.

    To see that you can achieve this lower bound, we can construct a Fubini ranking $f$ with content $\textbf{a}=(a_1,a_2,\ldots,a_n)$ and $\max{\textbf{a}}$ runs. 
    We begin by first constructing the initial run of $f$ by placing one element of each positive value in $\textbf{a}$, arranged in increasing order. 
    Then, we update the content to be $\textbf{a}'=(\max(0,a_1-1),\max(a_2-1,0),\ldots,\max(a_n-1,0)$, where we subtract $1$ from each nonzero entry in the content $\textbf{a}$. 
    We build the next run of $f$ by placing one element of each positive value in $\textbf{a}'$, arranged in increasing order. 
    Continue in this way, updating the content and constructing a new run of $f$. 
    The result of this iterative process ensures that $f$ has $\max\textbf{a}$ runs, at which point we have used all of the elements in the original content $\textbf{a}$.
\end{proof}






\Cref{tab:WFSR} provides data on the number of weakly flattened Fubini rankings with $n$ competitors and $k$ runs of ascents, which we denote by $\wflatruns_k(\FR_n)$. We remark that it is an open problem to determine formulas for $\wflatwruns_k(\FR_n)$. 
However, the row sums are given by the formula in \Cref{wsenum}.

\begin{table}[h]
\centering
\begin{tabular}{|c||c|c|c|c|c|c|c|}
\hline
$n\backslash k$ & 1 & 2 & 3 & 4 & 5 & 6 & 7 \\
\hline\hline
1 & 1 & 0 & 0 & 0 & 0 & 0 & 0 \\
\hline
2 & 1 & 1 & 0 & 0 & 0 & 0 & 0 \\
\hline
3 & 1 & 4 & 1 & 0 & 0 & 0 & 0 \\
\hline
4 & 1 & 13 & 7 & 1 & 0 & 0 & 0 \\
\hline
5 & 1 & 36 & 50 & 10 & 1 & 0 & 0 \\
\hline
6 & 1 & 93 & 286 & 112 & 13 & 1 & 0 \\
\hline
7 & 1 & 232 & 1419 & 1082 & 199 & 16 & 1 \\
\hline
\end{tabular}
\caption{Weakly flattened Fubini rankings with $k$ runs of ascents.}

\label{tab:WFSR}
\end{table}

\subsection{Flattened with runs of weak ascents}

Recall that \Cref{runOne}, \Cref{tooManyRuns}, and \Cref{thm:wfFub} all hold for flattened with runs of weak ascents. In addition, \Cref{lem:valid content} holds in the case of flattened with runs of weak ascents which motivates the following.

\begin{lemma}
\label{singlerun}
Consider the subset of \emph{$\flatwruns(\FR_n;\textbf{a})\subseteq\flatwruns(\FR_n)$} of flattened Fubini rankings with $n$ competitors,  runs of weak ascents, and content ${\bf a} = (a_1,a_2,\ldots,a_n)$. 
Then \emph{$|\flatwruns_1(\FR_n;\textbf{a})|=1$}.
\end{lemma}

\begin{proof} This result is equivalent showing that there exists a unique flattened Fubini ranking with $n$ competitors, runs of weak ascents, having content $\textbf{a}$, with exactly one run.
    Since runs only need to ascend weakly, we must place all the elements of the Fubini ranking, specified by the content $\textbf{a}$, in weakly increasing order to create a single run. Since this is unique for each content ${\bf a}$, $|\flatwruns_1(\FR_n;\textbf{a})|=1$ as desired. 
\end{proof}



\begin{lemma}\label{lem:max_runs_flatwruns}
Consider the subset of \emph{$\flatwruns(\FR_n;\textbf{a})\subseteq\flatwruns(\FR_n)$} of flattened Fubini rankings with $n$ competitors, runs of weak ascents, and reduced content ${\bf a} = (a_1,a_2,\ldots,a_m)$. 
If \emph{$\flatwruns_k(\FR_n;\textbf{a})$} is nonempty, then $k<\max(m,2)$.
\end{lemma}

\begin{proof}
    Since the Fubini ranking is flattened, the leading terms of the runs must all be distinct and appear in increasing order.
If $m=1$, then the Fubini ranking is the all ones tuple, which has exactly one run and so $k=1< \max(m,2)=\max(1,2)=2$.
Now consider $m>1$.    
Since there are $m$ distinct values in the content $\textbf{a}$, and the largest value of $\textbf{a}$ cannot start a run (since the runs are made up of weak ascents), we must have that $k< \max(2,m)$.
    \end{proof}

Note that we have not guaranteed that a flattened Fubini ranking with runs of weak ascents with the maximum possible number of runs will exist. So the bound in \Cref{lem:max_runs_flatwruns} may not be strict.

We conclude with \Cref{tab:SFWR} where we provide data on the number of elements in $\flatwruns_k(\FR_n)$, which is the set of flattened Fubini rankings with $n$ competitors and $k$ runs of weak ascents. We remark that it is an open problem to determine formulas for $|\flatwruns_k(\FR_n)|$, and for the row sums in the table which gives the number of elements in $\flatwruns(\FR_n)$, which is the set of flattened Fubini rankings with $n$ competitors.

\begin{table}[h]
\centering
\begin{tabular}{|c||c|c|c|c|}
\hline
$n\backslash k$ & 1 & 2 & 3 & 4 \\
\hline\hline
1 & 1 & 0 & 0 & 0 \\
\hline
2 & 2 & 0 & 0 & 0 \\
\hline
3 & 4 & 1 & 0 & 0 \\
\hline
4 & 8 & 9 & 0 & 0 \\
\hline
5 & 16 & 51 & 5 & 0 \\
\hline
6 & 32 & 235 & 86 & 0 \\
\hline
7 & 64 & 967 & 871 & 41 \\
\hline
\end{tabular}
\caption{Flattened Fubini rankings with $n$ competitors and $k$ runs of weak ascents.}
\label{tab:SFWR}
\end{table}





\section{Future work}\label{sec:future}

We conclude with the following directions for future work. First, we recall open problems that were included in previous sections.

\begin{enumerate}
\item Proposition~\ref{prop:end_in_1} relied on the set $\mathcal{B}_k(\bf{a})$of $n\times (k-1)$ nonnegative integer valued matrices whose $i$th row is a weak composition of $a_i$ and whose column sums are always greater than or equal to $1$. It is still an open question to determine the number of such matrices for all $n$ and $k$. 
\item 
Conjecture~\ref{genfunc} states the following:
Let 
$s_n=|\flatruns(\FR_n)|$ denote the number
of flattened Fubini rankings of length $n$ with runs of ascents. Then 
the generating function for $s_n$, is given by 
\[\sum_{n\geq 0}s_nx^n=\sum_{n\geq 0} \frac{1} { 2^{n+1}} \left( \prod_{i=1}^{n}( 1 + x(1+x)^i )\right).\]
This implies that the number of flattened Fubini rankings with runs of ascents appears to agree with \cite[\seqnum{A338793}]{OEIS},
for which no combinatorial description is provided. 
Thus, a proof to \Cref{genfunc} would yield a combinatorial interpretation for this sequence.
    \item Conjecture~\ref{conj:2} states the following: Let $s(n,k)=|\flatruns_k(\FR_n)|$ be the number of flattened Fubini rankings on $n$ competitors with $k$ runs of ascents. Then
       
\begin{align*}\label{eq:strong flat with nk}
s(2j+1,j-1)=\sum_{i=0}^{j-1} E_2(2j+1,i)2^{i},\end{align*}
    where $E_2(a, b)$ are the second-order Eulerian numbers \cite[\seqnum{A340556}]{OEIS}. 




    \item Tables~\ref{tab:weakly flattened Fubini with k runs},~\ref{tab:WFSR}, and~\ref{tab:SFWR} all have open enumeration questions that would be of interest, though there are no formal conjectures as to specifically desired formulas.
\end{enumerate}

Other directions and questions we have not discussed previously that we plan to pursue in future work:


\begin{enumerate}

    \item[(5)] Determining a strict upper bound for the number of runs in a strictly flattened Fubini ranking with runs of weak ascents and having a given content.
    \item[(6)] Determining strict upper and lower bounds for the number of runs in a flattened Fubini ranking with runs of ascents and having a given content.
    \item[(7)] Fix a content that leads to a flattened Fubini ranking with runs of ascents. How many rearrangements of the entries will yield another type of flattened Fubini ranking? And does this problem become easier if we enumerate these by the number of runs?
    \item[(8)] How many Fubini rankings with exactly $j$ ones exist under the various flattened definitions?
        \item[(9)] We could relax the flattened condition and let $\run_k(\FR_n)$ be the set of Fubini rankings with $k$ runs of ascents, and $\wrun_k(\FR_n)$ be the set of Fubini rankings with $k$ runs of weak ascents. We would like to give enumerations of these sets for any $1\leq k\leq n$.

\end{enumerate}




\bibliographystyle{plain}
\bibliography{bibliography}

\begin{thebibliography}{10}

\bibitem{baril2024flattenedcatalanwords}
Jean-Luc Baril, Pamela~E. Harris, and José~L. Ramírez.
\newblock Flattened catalan words, 2024.
\newblock To appear.

\bibitem{Beyene}
Fufa Beyene and Roberto Mantaci.
\newblock Flattened partitions and subxceedant functions.
\newblock {\em Journal of Integer Sequences}, 25(7):Art. 22.7.6, 2022.

\bibitem{unit_pf}
S.~Alex Bradt, Jennifer Elder, Pamela~E. Harris, Gordon Rojas~Kirby, Eva Reutercrona, Yuxuan~(Susan) Wang, and Juliet Whidden.
\newblock Unit interval parking functions and the r-{F}ubini numbers.
\newblock {\em La Mathematica}, 2024.

\bibitem{flat_sp}
Adam Buck, Jennifer Elder, Azia Figueroa, Pamela~E. Harris, Kimberly Harry, and Anthony Simpson.
\newblock Flattened stirling permutations.
\newblock {\em Integers}, 24, 2024.

\bibitem{callan}
David Callan.
\newblock Pattern avoidance in “flattened” partitions.
\newblock {\em Discrete Mathematics}, 309(12):4187--4191, 2009.

\bibitem{cayley_2009}
Arthur Cayley.
\newblock {\em On the analytical forms called Trees, with application to the theory of chemical combinations}, volume~9 of {\em Cambridge Library Collection - Mathematics}, page 427–460.
\newblock Cambridge University Press, 2009.

\bibitem{cruz2024discretestatisticsparkingfunctions}
Ari Cruz, Pamela~E. Harris, Kimberly~J. Harry, Jan Kretschmann, Matt McClinton, Alex Moon, John~O. Museus, and Eric Redmon.
\newblock On some discrete statistics of parking functions.
\newblock {\em J. Integer Seq.}, 8(6):Art. 20.9.6, 14, 2024.

\bibitem{elder2023boolean}
Jennifer Elder, Pamela~E. Harris, Jan Kretschmann, and J.~Carlos Mart\'{i}nez~Mori.
\newblock {P}arking functions, {F}ubini rankngs, and {B}oolean intervals in the weak order of $\mathfrak{S}_n$.
\newblock {\em Journal of Combinatorics}, 16(1), 2025.

\bibitem{flat_pf}
Jennifer Elder, Pamela~E. Harris, Zoe Markman, Izah Tahir, and Amanda Verga.
\newblock On flattened parking functions.
\newblock {\em Journal of Integer Sequences}, 26(5), 2023.

\bibitem{hadaway2022combinatorial}
Kimberly~P. Hadaway.
\newblock On combinatorical problems of generalized parking functions.
\newblock Williams College, Honors Thesis, 2022.

\bibitem{ONFRAB}
Olivia Nabawanda, Fanja Rakotondrajao, and Alex~Samuel Bamunoba.
\newblock Run distribution over flattened partitions.
\newblock {\em J. Integer Seq.}, 23(9):Art. 20.9.6, 14, 2020.

\bibitem{OEIS}
{OEIS Foundation Inc.}
\newblock The {O}n-{L}ine {E}ncyclopedia of {I}nteger {S}equences, 2023.
\newblock Published electronically at \url{http://oeis.org}.

\bibitem{summer_report}
Eva Reutercrona, Susan Wang, and Juliet Whiden.
\newblock Parking functions with fixed ascent and descent sets, 2022.
\newblock End of Summer Report for Summer@ICERM 2022.

\end{thebibliography}

\end{document}